\documentclass[a4paper]{amsart}
\usepackage{a4wide}
\usepackage{graphicx}
\usepackage{amsmath}
\usepackage{amssymb}
\usepackage{amsfonts}
\usepackage{amsthm}
\usepackage{eucal}
\usepackage{tikz}
\usetikzlibrary{trees, positioning, calc, cd}
\usepackage{hyperref}
\usepackage[capitalize,noabbrev]{cleveref}

\newcommand{\mycolim}[1]{\mathbin{\operatorname*{colim \hspace{1pt}}_{#1}^{}}}
\theoremstyle{plain}
\newtheorem{theorem}{Theorem}[section]

\newtheorem{lemma}[theorem]{Lemma}
\newtheorem{proposition}[theorem]{Proposition}

\newtheorem{corollary}[theorem]{Corollary}

\newtheorem*{proposition*}{Proposition}
\newtheorem*{corollary*}{Corollary}
\newtheorem*{theorem*}{Theorem}

\theoremstyle{plain}
\newcounter{zaehler}

\newtheorem*{introcor*}{Corollary}

\theoremstyle{definition}
\newtheorem{definition}[theorem]{Definition}

\newtheorem{remark}[theorem]{Remark}
\newtheorem{construction}[theorem]{Construction}
\newtheorem{notation}[theorem]{Notation}

\DeclareMathOperator{\sSeq}{sSeq}
\DeclareMathOperator{\cof}{cof}
\DeclareMathOperator{\CE}{CE}
\DeclareMathOperator{\forget}{forget}
\DeclareMathOperator{\triv}{triv}
\DeclareMathOperator{\colim}{colim}
\newcommand{\E}{\mathbb{E}}

\newcommand{\free}{\mathrm{free}}
\newcommand{\Barc}{\mathrm{Bar}}
\newcommand{\Alg}{\mathrm{Alg}}

\renewcommand{\nu}{\mathrm{nu}}

\newcommand{\End}{\mathrm{End}}
\newcommand{\id}{\mathrm{id}}

\newcommand{\gr}{\mathrm{gr}}

\newcommand{\Sp}{\mathrm{Sp}}

\newcommand{\Monad}{\mathrm{Monad}}

\newcommand{\Fun}{\mathrm{Fun}}
\newcommand{\Fil}{\mathrm{fil}}

\newcommand{\FB}{\mathrm{Fin}^\simeq}

\newcommand{\Op}{\mathrm{Op}}

\newcommand{\Ass}{\mathbf{Ass}}
 
\newcommand{\Com}{\mathbf{Com}}
\newcommand{\lev}{\otimes_{\mathrm{lev}}}

\makeatletter
\@addtoreset{theorem}{section}
\@addtoreset{subsection}{section}
\makeatother

\begin{document}

\title{Poincar\'{e}--Birkhoff--Witt theorems in higher algebra}

\author[O.~Antolin Camarena]{Omar Antolin Camarena} 
\address{National Autonomous University of Mexico }
\email{omar@matem.unam.mx}

\author[D.L.B.~Brantner]{Lukas Brantner} 
\address{Oxford University, Universit\'{e} Paris–Saclay (CNRS)}
\email{lukas.brantner@maths.ox.ac.uk, lukas.brantner@universite-paris-saclay.fr}

\author[G.S.K.S.~Heuts]{Gijs Heuts} 
\address{Utrecht University }
\email{g.s.k.s.heuts@uu.nl}

\begin{abstract}
We extend the classical Poincar\'{e}--Birkhoff--Witt theorem to higher algebra by establishing a version that applies to spectral Lie algebras. We deduce this statement from a basic relation between operads in spectra: the commutative operad is the quotient of the associative operad by a right action of the spectral Lie operad. This statement, in turn, is a consequence of a fundamental relation between different $\E_n$-operads, which we articulate and prove. We deduce a variant of the Poincar\'{e}--Birkhoff--Witt theorem for  relative enveloping algebras of $\E_n$-algebras. Our methods also give a simple construction and description of the higher enveloping $\E_n$-algebras of a spectral Lie algebra.
\end{abstract}

\maketitle 

\tableofcontents

\section{Introduction}

The classical Poincar\'{e}--Birkhoff--Witt theorem states that the universal enveloping algebra of a Lie algebra $\mathfrak{g}$ admits a natural filtration whose associated graded is the symmetric algebra on the underlying vector space of $\mathfrak{g}$. This statement may be phrased in the language of operads; it is essentially equivalent to the fact that the quotient of the associative operad by the right action of the Lie operad on it is the commutative operad. 

In this paper, we provide an extension of this statement to the world of higher algebra (see \Cref{thm:LieAssCom}), replacing the classical associative, Lie, and commutative operads by their lifts to the $\infty$-category of spectra. Correspondingly, we deduce a Poincar\'{e}--Birkhoff--Witt theorem for spectral Lie algebras (\Cref{PBWfull}). In fact, these results are consequences of a more basic relation between different $\E_n$-operads, namely \Cref{thm:En}, which has no `classical' analog. In particular we deduce a result in the style of Poincar\'{e}--Birkhoff--Witt for $\E_n$-algebras, which we make explicit in \Cref{PBWEn}. Finally, our methods provide a straightforward construction of the higher enveloping algebra of a spectral Lie algebra $\mathfrak{g}$. In \Cref{thm:Un} we show that the $\E_n$-enveloping algebra of $\mathfrak{g}$ may be calculated as the Chevalley--Eilenberg homology of the $n$-fold loop object of $\mathfrak{g}$, showing that our construction agrees with the one of  Knudsen \cite{knudsen2018higher} which relies on factorisation homology.

We will now review the results of this paper in more detail. The associative and commutative operads have evident lifts to the world of higher algebra; indeed, they can easily be defined in the $\infty$-category of spectra (or any symmetric monoidal $\infty$-category with coproducts). The case of the Lie operad is less straightforward, but Salvatore \cite{salvatore1998configuration} and Ching \cite{ching2005bar}  have shown that it also admits a lift to the $\infty$-category of spectra. Indeed, if we define $\mathbb{L}$ to be the Koszul dual $K(\mathbf{Com})$ of the commutative operad, then the operadic suspension $s\mathbb{L}$ (see \Cref{sec:symseq}) is an operad whose homology is concentrated in degree 0 and reproduces the classical Lie operad in abelian groups. Moreover, this construction recovers the more classical Koszul duality between the commutative and Lie operads in the derived  $\infty$-category $\mathrm{Mod}_R$ of any commutative ring $R$. 

To explain our first result, recall that operads in the $\infty$-category $\Sp$ of spectra may be defined as associative algebra objects in the $\infty$-category of symmetric sequences in spectra (cf.\ \cite{kelly2005operads, trimblenotes} \cite[4.1.2]{brantnerthesis}). 
 The  monoidal structure is given by the \emph{composition product} (see \Cref{sec:symseq} for a brief review). In particular, given a map of operads $\mathcal{O} \to \mathcal{P}$, we may  interpret $\mathcal{P}$ as an $\mathcal{O}$-bimodule in symmetric sequences.

\begin{notation} In what follows, we will denote the associative operad $\Ass$ by $\E_1$ and the commutative operad $\Com$ by $\E_\infty$, to make the notation consistent with our later results on $\E_n$-operads. We take all of our operads to be \emph{nonunital} (without emphasising this in the notation), meaning that their term of 0-ary operations is trivial. We will denote by $\mathbf{1}$ the trivial operad, which has a unit in arity 1 and nothing else. This is also the nonunital $\E_0$-operad.
\end{notation}  

We can now state our first main theorem:
\begin{theorem} \label{PBWtriv}
\label{thm:LieAssCom}
There is a commutative square of operads in the $\infty$-category $\Sp$ as follows:
\[
\begin{tikzcd}
s\mathbb{L} \ar{d}\ar{r} & \mathbf{1} \ar{d} \\
\E_1 \ar{r} & \E_\infty.
\end{tikzcd}
\]
Moreover, the induced map $\E_1 \circ_{s\mathbb{L}} \mathbf{1} \to \E_\infty$ is an equivalence of left $\E_1$-modules.
\end{theorem}

\begin{remark}
In the statement of the theorem, $\E_1 \circ_{s\mathbb{L}} \mathbf{1}$ denotes the \emph{relative  composition product}  of $\E_1$ and $\mathbf{1}$ over ${s\mathbb{L}}$. For an operad $\mathcal{O}$ with a right module $\mathcal{M}$ and a left module $\mathcal{N}$, such a relative composition product may be computed via the bar construction
\[
\mathcal{M} \circ_{\mathcal{O}} \mathcal{N} = \mathrm{colim}_{\mathbf{\Delta}^{\mathrm{op}}} (\mathcal{M} \circ {\mathcal{O}}^{\circ \bullet} \circ \mathcal{N}).
\]
\end{remark}

\begin{remark}
The square of \Cref{thm:LieAssCom} is \emph{not} a pushout square of operads (cf.\ \Cref{prop:notpushouts}).  
\end{remark}

In fact, \Cref{thm:LieAssCom} will follow as a limiting case of a statement about the relation between different $\E_n$-operads (see \Cref{thm:En} below).
 To be precise, if $\E_n$ denotes the usual (nonunital) operad of little $n$-cubes in the $\infty$-category of spaces, then we will be interested in the operad $\Sigma^\infty_+ \E_n$ in spectra. Throughout this paper, we will use the short-hand $\E_n$ for this nonunital, stable version of the $\E_n$-operad. Given integers $m,n \geq 0$, we will generically use the letter $\iota$ to denote the usual morphism $$\E_m \to \E_{m+n}$$ arising from the standard inclusion $\mathbb{R}^m \to \mathbb{R}^{m+n}$ of Euclidean spaces. \Cref{PBWtriv} motivates the following:  
 
\begin{definition}[Composition squares] A commutative square of operads 
\[
\begin{tikzcd}
\mathcal{O} \ar{r}\ar{d} & \mathcal{P} \ar{d} \\
\mathcal{Q} \ar{r} & \mathcal{R}
\end{tikzcd}
\]
is called a \emph{composition square} if 
 the induced map $\mathcal{Q} \circ_{\mathcal{O}} \mathcal{P} \to \mathcal{R}$ is an equivalence (of $\mathcal{Q}$-$\mathcal{P}$-bimodules).
\end{definition} 

We then establish the following fundamental relation between different $\E_n$-operads, with $s$ denoting operadic suspension (see \Cref{sec:symseq}):

\begin{theorem}
\label{thm:En}
Given integers $k,m,n \geq 0$ there is a composition square of operads in spectra 
\[
\begin{tikzcd}
\E_{k+m} \ar{d}{\beta}\ar{r}{\iota} & \E_{k+m+n} \ar{d}{\beta} \\
s^k\E_{m} \ar{r}{\iota} & s^k\E_{m+n}.
\end{tikzcd}
\] 
\end{theorem}

We will formally describe the morphism $\beta\colon \E_{m+k} \to s^k\E_m$ in detail in \Cref{sec:mainproof}. For now, let us outline two ways of thinking about it. Without loss of generality we take $k=1$ -- the general case is obtained by a $k$-fold composition of such maps. First, $\beta$ can be obtained as the \emph{Koszul dual} of the morphism $\iota\colon \E_m \to \E_{m+1}$, relying on the fact that the Koszul dual of the operad $\E_m$ is the operad $s^{-m} \E_m$, see Remark \ref{rmk:betaKD}. Secondly, under the `degree shifting' equivalence $\mathrm{Alg}_{s\E_m}(\Sp) \simeq \mathrm{Alg}_{\E_m}(\Sp),$ the morphism $\beta$ induces a left adjoint \vspace{-1pt}  functor 
\[
\beta_!\colon \mathrm{Alg}_{\E_{m+1}}(\Sp) \to \mathrm{Alg}_{\E_m}(\Sp).
\]
which may be identified with the usual bar construction. This second perspective is the one used to construct and characterise $\beta$ in \cite{heutsland}.

In addition to \Cref{thm:LieAssCom}, we will deduce further limiting cases of \Cref{thm:En}:
\begin{theorem}
\label{thm:specialcases}
The following are composition squares of operads, where $\sigma$ denotes the suspension map (see Section \vspace{-1pt}  \ref{sec:symseq}):
\[
\begin{tikzcd}
\mathbb{L} \ar{d}\ar{r}{\sigma^n} & s^n\mathbb{L} \ar{d}{\beta} && \E_n \ar{d}\ar{r}{\iota} & \E_\infty \ar{d}{\sigma^n} \\
\mathbf{1} \ar{r} & \E_n && \mathbf{1} \ar{r} & s^n\E_\infty.
\end{tikzcd}
\]
\end{theorem}

The morphism $\beta\colon s^n\mathbb{L} \to \E_n$ featuring in \Cref{thm:specialcases} can be obtained as an inverse limit over $k$ of the morphisms $\beta\colon s^{-k} \E_{n+k} \to \E_n$. Alternatively, it can be described as (a shift of) the Koszul dual of the inclusion $\iota\colon \E_n \to \E_{\infty}$, see Remark \ref{rmk:betaKD}. Identifying the $\infty$-categories $\Alg_{s^n\mathbb{L}}(\Sp)$ and $\Alg_{\mathbb{L}}(\Sp)$ via an $n$-fold degree shift, this morphism induces a left adjoint functor
\[
\beta_!\colon \Alg_{\mathbb{L}}(\Sp) \to \Alg_{\E_n}(\Sp) 
\]
which we refer to as the \emph{$\E_n$-enveloping algebra} functor and denote by $U_n$. 

This construction of $U_n$ has also been sketched  by Ayala--Francis \cite{ayala2015factorization}, modulo the definition of the morphism $\beta\colon s^n\mathbb{L} \to \E_n$, which was only supplied more recently by Salvatore--Ching \cite{ching2022koszul}. A different construction via factorisation homology has been implemented by Knudsen \cite{knudsen2018higher}. In the special case $n=1$, the functor $U_1$ can be thought of as a lift of the classical universal enveloping algebra functor to spectral Lie algebras. The degenerate case $n=0$ gives the pushforward along the augmentation map $\mathbb{L} \to \mathbf{1}$, which shall be denoted by $\CE$ as it is an enhancement of the classical Chevalley--Eilenberg functor.

We will show that the first composition square of \Cref{thm:specialcases} easily leads to \Cref{thm:Un} below. In particular, it shows that our $U_n$ is naturally equivalent to Knudsen's, since the main theorem of \cite{knudsen2018higher} provides the same description for his functor.

\begin{theorem}
\label{thm:Un}
Given a spectral Lie algebra $\mathfrak{g} \in \Alg_{\mathbb{L}}(\Sp)$ there is a natural equivalence \vspace{-1pt} 
\[
U_n(\mathfrak{g}) \simeq \CE(\Omega^n \mathfrak{g}).
\]
\end{theorem}

\begin{remark}
\label{rmk:Knudsenspseq}
The first square of \Cref{thm:specialcases} gives an equivalence $\E_n \simeq \mathbf{1} \circ_{\mathbb{L}} s^n\mathbb{L}$. For a given spectrum $X$, the skeletal filtration on the bar construction $\mathbf{1} \circ_{\mathbb{L}} s^n\mathbb{L}$ therefore gives a spectral sequence converging to $\free_{\E_n}(X)$, which is the one used by Brantner--Hahn--Knudsen \cite{brantnerhahnknudsen} and Zhang \cite{zhang2021quillen} to study the (generalised) homology of $\E_n$-algebras.
\end{remark}

\Cref{PBWtriv} and \Cref{thm:En} imply the following analogues of the classical PBW theorem via a filtration trick we learned from \cite{gaitsgory2017study}. For a symmetric sequence $A$ we write $\free_A$ for the functor $X \mapsto \bigoplus_{n \geq 0} (A(n) \otimes X^{\otimes n})_{h\Sigma_n}$. Let $A^{\vee}$ denote  the termwise Spanier--Whitehead dual of the symmetric sequence $A$. Writing $U=U_1$ for the universal $\E_1$-enveloping algebra, we have:

\begin{corollary}[PBW theorem for spectral Lie algebras] \label{PBWfull}
	Given a spectral Lie algebra $\mathfrak{g}$, the universal enveloping algebra $U(\mathfrak{g})$ admits a natural exhaustive filtration of which the associated graded spectrum admits a natural equivalence
	\vspace{-1pt} 
	\[ \gr(U(\mathfrak{g}))\simeq \free_{\E_\infty}(\Sigma^{-1}\forget(\mathfrak{g})). \]
\end{corollary}
\begin{remark}
The degree shift on the right-hand side of the equivalence above stems from the fact the spectral Lie operad $\mathbb{L}$ produces a bracket of degree -1, as compared to the classical Lie operad with bracket in degree 0. This is also reflected in the fact that the homology of a spectral Lie algebra is a \emph{shifted} Lie algebra.
By an exhaustive filtration on $U(\mathfrak{g})$, we mean a sequence of spectra $F_0 U(\mathfrak{g})  \rightarrow  F_1 U(\mathfrak{g})   \rightarrow \ldots$ with colimit $U(\mathfrak{g})$, and the associated graded is given by the sum of the cofibres of all these maps.
\end{remark}

Given  $0 \leq m \leq n$ we construct a version of the enveloping algebra functor  from $\E_n$-algebras to $\E_m$-algebras in 
\Cref{relenv}. It satisfies the following version of the PBW theorem:
\begin{corollary}[PBW theorem for $\E_n$-algebras] \label{PBWEn}
Write $k=n-m$.	Given an $\E_n$-algebra  $A$, the relative enveloping algebra $U_{n,m} (A) \in \Alg_{s^{k}\E_{m}}(\Sp)$ admits a natural exhaustive filtration of which the associated graded spectrum admits a natural equivalence
	\[\gr(U_{n,m} (A))\simeq
	\free_{s^{k} \E_{k}^\vee}(\forget(A)).\]
	Equivalently, the $k$-fold bar construction $\Barc^{k} A$ admits an exhaustive filtration\vspace{-1pt} \mbox{with associated graded}
	\[\gr(\Barc^{k} A) \simeq \free_{\E_{k}^\vee}(\Sigma^{k} \forget(A)). \]
\end{corollary}

\subsection*{Acknowledgements}

We are grateful to Jacob Lurie  for several helpful discussions. L.B. was supported  by a Royal Society University Research Fellowship at Oxford   (URF\textbackslash R1\textbackslash 211075) and by the CNRS at Orsay. G.H. was supported by an ERC Starting Grant (no.\ 950048) and an NWO VIDI Grant (no.\ 223.093).

\section{Symmetric sequences and operads}
\label{sec:symseq}

In this section, we give a brisk review of the basic facts about symmetric sequences and operads we shall require. We take the point of view that an operad is an algebra for the monoidal structure given by the composition product on the category of symmetric sequences,
following the 1-categorical argument of Carboni described by Kelly \cite{kelly2005operads} and Trimble \cite{trimblenotes}, as well as its $\infty$-categorical generalisation (cf.\  \cite[4.1.2]{brantnerthesis}\cite[3.1]{Brantner}).

To make this precise, write \(\FB\) for  the groupoid of finite sets and bijections. \begin{definition} Let \(\mathcal{C}\) be a  symmetric monoidal \(\infty\)-category. The \(\infty\)-category of \emph{symmetric sequences} in \(\mathcal{C}\) is given by $$\sSeq(\mathcal{C}) := \Fun(\FB, \mathcal{C}).$$ The value of a symmetric sequence \(A\) on a set with \(n\) elements will be denoted by \(A(n)\) and it is an object of \(\mathcal{C}\) equipped with an action of the symmetric group \(\Sigma_n\). \end{definition}

In this paper, it will be sufficient to consider the case where \(\mathcal{C}\) is the $\infty$-category \(\Sp\) of spectra or minor variants thereof, although most of what we say will go through in much greater generality. We will need to consider three different monoidal structures on the $\infty$-category \(\sSeq(\Sp)\).

\subsubsection*{Day convolution product}
The first is given by \textit{Day convolution} (cf.\  \cite{glasman2016day} \cite[2.2.6]{HA}) based on the disjoint union monoidal structure on \(\FB\). Explicitly,  it is determined by the formula \[(A \otimes B)(n) = \bigoplus_{a+b=n} \Sigma^\infty_+ \Sigma_n 
 \otimes_{h(\Sigma_a \times \Sigma_b)}A(a) \otimes B(b);\]
 here $\oplus$ and $\otimes$ denote the wedge and smash product in spectra, respectively.

The fully faithful functor \(\iota : \Sp \to \sSeq(\Sp)\) given by \(\iota(X)(0) = X\) and \(\iota(X)(n)=0\) for \(n>0\) refines to a symmetric monoidal functor
 for this Day convolution structure, and we obtain a tensoring of \(\sSeq(\Sp)\) over \(\Sp\).

\subsubsection*{Composition product} The second monoidal structure is the \emph{composition product}; on objects, the composition product is given by
\[ A \circ B = \bigoplus_{n \ge 0} \left( A(n) \otimes B^{\otimes n}  \right)_{h \Sigma_{n}}.\]
One can check that for a spectrum \(X\) and symmetric sequence \(A\), the composition \(A \circ \iota(X)\) is concentrated in degree \(0\), that is, it lies in the essential image of \(\iota\). Therefore, for a fixed \(A\) there is a functor \(\free_A \in \End(\Sp)\) satisfying \(\iota(\free_A(X)) = A \circ \iota(X)\). Explicitly, it is given by
\[\free_A(X) = \bigoplus_{n \ge 0} \left( A(n) \otimes X^{\otimes n} \right)_{h\Sigma_n}.\]
Since this action of \(A\) on spectra is given simply by restricting the composition product to the essential image of \(\iota\), the functor \(\free\colon \sSeq(\Sp) \to \End(\Sp)\) refines to a  monoidal functor, where \(\End(\Sp)\) is endowed with the monoidal structure given by composition of functors.

\begin{definition}\label{operadsinsp}
\emph{Operads} in spectra are defined to be algebra objects in the monoidal \(\infty\)-category \((\sSeq(\Sp), \circ)\), and the $\infty$-category they form is denoted by \(\Op(\Sp) = \Alg(\sSeq(\Sp))\). 
\end{definition} 
The resulting $\infty$-category of operads underlies the model category of operads in $S$-modules (cf.\ e.g.\ \cite[5.4.2]{brantnerthesis}).
The free functor described above therefore induces a functor $$\free\colon \Op(\Sp) \to \Monad(\Sp).$$ One can consider left modules, right modules, and bimodules over an operad \(\mathcal{O}\) in \(\sSeq(\mathcal{O})\); a left module of the form \(\iota(X)\) is equivalently an algebra for the monad \(\free_{\mathcal{O}}\) and one says that \(X\) is an \emph{\(\mathcal{O}\)-algebra}. We denote the $\infty$-category of $\mathcal{O}$-algebras by $\mathrm{Alg}_{\mathcal{O}}(\Sp)$ or simply $\mathrm{Alg}_{\mathcal{O}}$ if no confusion can arise.

\subsubsection*{Levelwise product}
The third and simplest monoidal structure on \(\sSeq(\Sp)\) we will need is the \emph{levelwise tensor product}, which is simply given by $$(A \lev B)(n) = A(n) \otimes B(n).$$ We will use this monoidal structure as an auxiliary tool in our discussion of the operadic suspension and the suspension morphism. The main property we will need is the fact that $$\lev : (\sSeq(\Sp) \times \sSeq(\Sp), \circ \times \circ) \to (\sSeq(\Sp), \circ)$$ has a lax monoidal structure \cite[Proposition 3.9]{Brantner}. As a consequence, \(\lev\) induces a functor \[\Op(\Sp) \times \Op(\Sp) \to \Op(\Sp)\] so that given two operads \(\mathcal{P}\) and \(\mathcal{Q}\), the symmetric sequence \(\mathcal{P} \lev \mathcal{Q}\) is again an operad. On the level of algebras, we obtain a functor \begin{equation} \label{algtens}
	\Alg_{\mathcal{P}} \times \Alg_{\mathcal{Q}} \rightarrow  \Alg_{\mathcal{P} \lev \mathcal{Q}} \end{equation} sending a pair $ (A, B)$ to $A \otimes B.$ This functor varies naturally in the pair $(\mathcal{P}, \mathcal{Q})$.

\subsubsection*{Suspension functor}
We now turn our attention to the operation of \emph{operadic suspension}. The suspension of a nonunital operad in spectra \(\mathcal{O}\) is an operad \(s\mathcal{O}\) such that the corresponding free algebra monads satisfy \[\free_{s \mathcal{O}} \simeq \Sigma^{-1} \circ \free_{\mathcal{O}} \circ \Sigma,\] and consequently, \(s\mathcal{O}\)-algebra structures on \(X\) are in one-to-one correspondence \(\mathcal{O}\)-algebra structures on \(\Sigma X\). The underlying symmetric sequence of the operadic suspension is given by \[(s\mathcal{O})(n) = \Sigma^{-1}(\mathbb{S}^1)^{\otimes n} \otimes \mathcal{O}(n),\]
where \(\mathbb{S}^1\) is the suspension of the sphere spectrum and the \(\Sigma_n\) action in the left factor permutes the \(\mathbb{S}^1\) factors and acts trivially on the suspension coordinate.

To describe the operad structure of \(s \mathcal{O}\),  we consider the endomorphism operad  $$s \E_\infty := \End(\mathbb{S}^{-1})$$
of \(\mathbb{S}^{-1}\); the endomorphism operad of a spectrum $X$ is the endomorphism object (\cite[4.7.1]{HA}) of $X$ with respect to the left tensoring of $\Sp$ over  $\sSeq(\Sp)$.
Our notation reflects the fact that it will be 
the operadic suspension of the commutative operad. 

The spectrum of \(n\)-ary operations in \(s \E_\infty\) is given by \(\mathrm{map}(\mathbb{S}^{-n},\mathbb{S}^{-1}) \simeq \mathbb{S}^{n-1}\), and tracing through the action of \(\Sigma_n\) on \(\mathbb{S}^{-n} \simeq (\mathbb{S}^{-1})^{\wedge n}\) we see that the action of \(\Sigma_n\) on \(s\E_\infty(n) \simeq \Sigma^{-1} (\mathbb{S}^1)^{\wedge n}\) is the one described above. (Equivalently, $s\E_{\infty}(n)$ is the representation sphere of $\rho$, the quotient of the standard $n$-dimensional permutation representation by its diagonal.) Using the levelwise tensor product, we define: 
\begin{definition} 
The operadic suspension functor is defined as  \[s := (-) \lev s\E_\infty  : \Op(\Sp) \to \Op(\Sp).\]
\end{definition}

On the level of algebras, tensoring with the canonical $s \E_\infty$-algebra $\mathbb{S}^{-1}$ (whose structure map \(s \E_\infty \to \End(\mathbb{S}^{-1})\) is the identity) gives a functor 
 \begin{equation} \label{tauto}\Alg_{\mathcal{O}}(\Sp)  \xrightarrow{}\Alg_{s\mathcal{O}}(\Sp), \ \  X \ \mapsto\  X  \otimes  \mathbb{S}^{-1} \end{equation}
via \eqref{algtens}. There is also an operadic desuspension operation $s^{-1}$ defined by tensoring with $\End(\mathbb{S}^1)$, which is inverse to $s$. On the level of algebras, this shows that the morphism in \eqref{tauto} is   an equivalence of $\infty$-categories.

\subsubsection*{Suspension morphism} Given an operad $ \mathcal{O}$ in spectra,
we will also need the \emph{suspension morphism} $$\sigma\colon \mathcal{O}  \to s \mathcal{O}.$$ In the special case of \(\mathcal{O} = \E_\infty\), we define the map \(\sigma\colon \E_\infty \to \End(\mathbb{S}^{-1})\) as the map endowing \(\mathbb{S}^{-1}\) with the $\E_\infty$-ring structure of the Spanier--Whitehead dual of $S^1$, which can also be described as  \(\Omega_{\E_\infty} \mathbb{S}^0\). Here \(\mathbb{S}^0\), being the monoidal unit of $\Sp$, is equipped with its canonical $\E_\infty$-algebra stucture and \(\Omega_{\E_\infty}\) is the loop functor on the $\infty$-category \(\Alg_{\E_\infty}(\Sp)\). Notice that by construction this \(\E_\infty\)-ring structure on \(\mathbb{S}^{-1}\) is the image under \(\sigma^{*}\) of the canonical \(s \E_\infty\)-algebra structure on \(\mathbb{S}^{-1}\) used above.
For a general operad \(\mathcal{O}\) the definition uses the the previous map \(\sigma\).

\begin{definition}
	Given a nonunital operad \(\mathcal{O}\)  in spectra, the suspension morphism is  the composite \[\sigma\colon \mathcal{O} \simeq \mathcal{O} \lev \E_\infty \xrightarrow{\id \lev \sigma} \mathcal{O} \lev s \E_\infty \simeq s \mathcal{O},\]
	where we have used the equivalence of operads $\mathcal{O} \simeq \mathcal{O} \lev \E_\infty$ (cf.\ e.g.\ \cite[Section 3.2]{Brantner}).
\end{definition}

Since $\mathcal{O}$ is nonunital (i.e. $\mathcal{O}(0) = 0$), the free \(\mathcal{O}\)-algebra on the zero spectrum \(0\) has underlying spectrum \(\free_{\mathcal{O}}(0) = 0\) and it is a zero object in \(\Alg_{\mathcal{O}}(\Sp)\). This implies that the loop functor \(\Omega_{\mathcal{O}} : \Alg_{\mathcal{O}}(\Sp) \to \Alg_{\mathcal{O}}(\Sp)\) and the forgetful functor \(\forget : \Alg_{\mathcal{O}}(\Sp) \to \Sp\) satisfy \[\forget \circ \Omega_{\mathcal{O}} \simeq \Omega \circ \forget.\]

\begin{proposition}
\label{prop:suspensionmorphism}
Given a nonunital operad $\mathcal{O}$, 
the composite 
$$ \Alg_{\mathcal{O}} \xrightarrow{(-) \otimes \mathbb{S}^{-1}} \Alg_{s\mathcal{O}} \xrightarrow{\sigma^{\ast}} \Alg_{\mathcal{O}}$$
is equivalent to the loops functor $\Omega_{\mathcal{O}}$.
Here $ (-) \otimes \mathbb{S}^{-1}$ is the equivalence in 
\eqref{tauto} and $\sigma^\ast$ is induced by restricting along the suspension morphism \(\sigma\colon \mathcal{O}  \xrightarrow{ }   s \mathcal{O}\).
\end{proposition}
\begin{proof}
	The equivalence of operads  $\mathcal{O} {\lev} \E_\infty \simeq \mathcal{O}$ gives a functor \[	\Alg_{\mathcal{O}} \times \Alg_{\E_\infty} \rightarrow  \Alg_{\mathcal{O}}\] 
	via \eqref{algtens}. Since the forgetful functor  creates limits, 
	 tensoring with the $\E_\infty$-ring spectrum \(\Omega_{\E_\infty} \mathbb{S}^0\) recovers the loops functor $\Omega_{\mathcal{O}}$.
The map of pairs of operads 
$$(\id,\sigma)\colon (\mathcal{O}, \E_\infty) \rightarrow(\mathcal{O}, s\E_\infty)$$	
gives rise to a commutative square of $\infty$-categories 
\[
\begin{tikzcd}
	 \Alg_{\mathcal{O}} \times \Alg_{s \E_\infty} \ar{r}{\id \times \sigma^\ast} \ar{d}{\otimes} & 	\Alg_{\mathcal{O}} \times \Alg_{\E_\infty} \ar{d}{\otimes} \\
	\Alg_{ \mathcal{O} \otimes s\E_\infty } \ar{r}{(\id \otimes \sigma)^{\ast}} & 	\Alg_{\mathcal{O} \otimes    \E_\infty}
\end{tikzcd}
\]
Recall that the map $\id \otimes \sigma$ on the bottom row is, by construction, equivalent to the map $\sigma\colon \mathcal{O} \to s\mathcal{O}$. Given an $\mathcal{O}$-algebra $X$, tracing the pair $(X, \mathbb{S}^{-1})$ through the square gives a natural equivalence $$ \sigma^{\ast}(X \otimes \mathbb{S}^{-1}) \simeq X \otimes \Omega_{\E_\infty} \mathbb{S}^0 \simeq \Omega_{\mathcal{O}} X.$$
\end{proof}

\begin{remark}
\label{rmk:suspensionHL}
Heuts and Land \cite{heutsland} construct a suspension morphism $\sigma\colon \E_n \to s\E_n$ in the specific case of the $\E_n$-operad and show that it is characterised essentially uniquely by the fact that it satisfies the conclusion of \Cref{prop:suspensionmorphism}. Since the morphism $\sigma$ that we constructed above also satisfies that conclusion, it follows that the two constructions of $\sigma$ agree when they are both defined.
\end{remark}

\section{Proof of the main result}\label{sec:mainproof}

The aim of this section is to prove \Cref{thm:En}. In the next section we work out several consequences and special cases, including \Cref{thm:LieAssCom}.

To prove \Cref{thm:En} it will suffice to establish the following special case: there exists a commutative square
\[
\begin{tikzcd}
\E_{n} \ar{d}{\beta}\ar{r}{\iota} & \E_{n+1} \ar{d}{\beta} \\
s\E_{n-1} \ar{r}{\iota} & s\E_{n}
\end{tikzcd}
\]
inducing an equivalence
\[
s\E_{n-1} \circ_{\E_{n}} \E_{n+1} \xrightarrow{\simeq } s\E_{n}
\] 
of $(s\E_{n-1},\E_{n+1})$-bimodules. Indeed, the general case of \Cref{thm:En} follows by composing copies of this basic square horizontally and/or vertically. The relevant square was constructed by Land and the third author in \cite[Theorem 3.11]{heutsland} (note that we are thinking of $\E_{n+1}$ as $\E_{1} \otimes \E_{1} \otimes \E_{n-1}$, using one of the $\E_1$-factors for the horizontal arrows $\iota$ and the other $\E_1$-factor for the vertical arrows $\beta$). To summarise, this goes as follows. First one establishes a commutative diagram of left adjoint functors
\begin{equation}
\label{eq:squareEn}
\begin{tikzcd}
\Alg_{\E_n}(\Sp) \ar{d}{\Barc} \ar{r}{\iota_!} & \Alg_{\E_{n+1}}(\Sp) \ar{d}{\Barc} \\
\Alg_{\E_{n-1}}(\Sp) \ar{r}{\iota_!} & \Alg_{\E_n}(\Sp).
\end{tikzcd}
\end{equation}
All of these $\infty$-categories are monadic over $\Sp$, giving a corresponding square of monads. The functor assigning to an operad in $\Sp$ its corresponding monad on $\Sp$ is fully faithful on the subcategory of $\E_n$-operads and their (de)suspensions \cite[Theorem 3.8]{heutsland}, so that the given square corresponds essentially uniquely to the desired square of operads. We record the following (which is part of the statement of \cite[Theorem 3.11]{heutsland}) for later use:

\begin{lemma}
\label{lem:betaiota}
In the square of operads above we have $\beta\circ\iota \simeq \sigma \simeq \iota\circ\beta$, where $\sigma\colon \E_n \to s\E_n$ denotes the suspension morphism.
\end{lemma}

\begin{remark}
\label{rmk:betaKD}
Let us write $K\colon \mathrm{Op}(\Sp) \to \mathrm{coOp}(\Sp)$ for the Koszul duality functor that takes an operad $\mathcal{O}$ first to its bar construction $\Barc(\mathcal{O})$ (which is a cooperad) and then to the termwise Spanier--Whitehead dual $\Barc(\mathcal{O})^{\vee}$, which is an operad. It is a theorem of Ching--Salvatore \cite{chingsalvatore} that $K\E_n \cong s^{-n} \E_n$, see also Malin's proof \cite{malin} or the forthcoming \cite{heutsland2} for an $\infty$-categorical version. Under this identification, the Koszul dual of the morphism $\iota\colon \E_n \to \E_{n+1}$ is (up to an $n+1$-fold shift) precisely the morphism $\beta\colon \E_{n+1} \to s\E_n$
featuring above. A proof of this fact will appear in \cite{heutsland2}.
\end{remark}

Before proving \Cref{thm:En} it will be convenient to establish a certain recognition criterion for composition squares. Let 
\begin{equation}
\label{eq:sqoper}
\begin{tikzcd}
\mathcal{O} \ar{r}{f}\ar{d}{g} & \mathcal{P} \ar{d}{g'} \\
\mathcal{Q} \ar{r}{f'} & \mathcal{R}
\end{tikzcd}
\end{equation}
be a commutative square of operads in spectra. It induces a corresponding square of left adjoint functors between algebra categories
\[
\begin{tikzcd}
\mathrm{Alg}_{\mathcal{O}}(\Sp) \ar{r}{f_!}\ar{d}{g_!} & \mathrm{Alg}_{\mathcal{P}}(\Sp) \ar{d}{g'_!} \\
\mathrm{Alg}_{\mathcal{Q}}(\Sp) \ar{r}{f'_!} & \mathrm{Alg}_{\mathcal{R}}(\Sp).
\end{tikzcd}
\]
Associated to this square is another, namely:
\begin{equation}
\label{eq:BCsquare}
\begin{tikzcd}
\mathrm{Alg}_{\mathcal{O}}(\Sp) \ar{d}{g_!}\ar[phantom, sloped]{dr}[description]{\Rightarrow} & \mathrm{Alg}_{\mathcal{P}}(\Sp) \ar{d}{g'_!}\ar{l}[swap]{f^*} \\
\mathrm{Alg}_{\mathcal{Q}}(\Sp) & \mathrm{Alg}_{\mathcal{R}}(\Sp) \ar{l}[swap]{(f')^*}.
\end{tikzcd}
\end{equation}
This last square is generally only lax commutative, in the sense that there is a natural transformation $g_! \circ f^* \Rightarrow (f')^*g'_!$. This natural transformation is the adjoint of the composite
\[
f'_! \circ g_! \circ f^* \cong g'_! \circ f_! \circ f^* \Rightarrow g'_!
\]
where the arrow arises from the counit of the adjoint pair $(f_!,f^*)$.

\begin{lemma}
\label{lem:recognitioncomposition}
The square of operads (\ref{eq:sqoper}) is a composition square if and only if the lax commutative square (\ref{eq:BCsquare}) is commutative, i.e., if the natural transformation $g_! \circ f^* \Rightarrow (f')^*g'_!$ is an equivalence.
\end{lemma}
\begin{proof}
For a $\mathcal{P}$-algebra $X$, the natural transformation of the lemma can be identified as the evident map
\[
\mathcal{Q} \circ_{\mathcal{O}} X \to \mathcal{R} \circ_{\mathcal{P}} X.
\]
On the left-hand side $X$ is implicitly regarded as an $\mathcal{O}$-algebra via $f^*$. Since both expressions preserve sifted colimits in the variable $X$, the map is an equivalence if and only if it is so in the special case of free $\mathcal{P}$-algebras $X = \mathcal{P} \circ A$ with $A \in \Sp$. Then it reduces to the natural map
\[
(\mathcal{Q} \circ_{\mathcal{O}} \mathcal{P}) \circ A \to \mathcal{R} \circ A.
\]
This is a natural equivalence if and only if the underlying map of symmetric sequences $\mathcal{Q} \circ_{\mathcal{O}} \mathcal{P} \to \mathcal{R}$ is an equivalence; indeed, the `if'-direction is clear, whereas the `only if' follows by taking the Goodwillie derivatives of the natural transformation above. To be precise, we apply Goodwillie calculus to the $\infty$-category $\mathrm{Fun}^{\omega}(\Sp,\Sp)$ of endofunctors preserving filtered colimits (see \cite[Section 6.1]{HA} and \cite[Section 3.1.3]{blansblom}) and use that for a symmetric sequence $F$, the derivatives of the functor $F \circ (-)$ are naturally equivalent to $F$ itself.
\end{proof}

\begin{proof}[Proof of \Cref{thm:En}]
In the commutative square (\ref{eq:squareEn}) we can take right adjoints of the horizontal functors to obtain the lax square
\[
\begin{tikzcd}
\Alg_{\E_n}(\Sp) \ar{d}{\Barc}\ar[phantom, sloped]{dr}[description]{\Rightarrow} & \Alg_{\E_{n+1}}(\Sp) \ar{d}{\Barc} \ar{l}[swap]{\iota^*} \\
\Alg_{\E_{n-1}}(\Sp) & \Alg_{\E_n}(\Sp) \ar{l}[swap]{\iota^*}.
\end{tikzcd}
\]
By Lemma \ref{lem:recognitioncomposition} it will suffice to show that this square in fact commutes (up to natural equivalence). But this is clear from the fact that $\iota^*$ preserves tensor products and sifted colimits, which are the two ingredients to form the bar construction.
\end{proof}

\section{Examples and higher enveloping algebras}

We start this section by deducing the composition squares of \Cref{thm:LieAssCom} and \Cref{thm:specialcases} and afterwards discuss the implications for higher enveloping algebras. In the following proofs, we will need the following observation on (co)limits of composition squares:

\begin{lemma}
\label{lem:gluingsquares}
A filtered colimit of composition squares is a composition square. Dually, consider a cofiltered diagram $\{S_i\}_{i \in I}$ of composition squares, with each $S_i$ of the form
\[
\begin{tikzcd}
\mathcal{O}_i \ar{r}\ar{d} & \mathcal{P}_i \ar{d} \\
\mathcal{Q}_i \ar{r} & \mathcal{R}_i.
\end{tikzcd}
\]
Suppose that $\mathcal{O}_i(1)$ is equivalent to the monoidal unit for all $i$ and similarly for $\mathcal{P}_i$ and $\mathcal{Q}_i$. Furthermore, suppose that the operads $\mathcal{O}_i$ and $\mathrm{lim}_{i \in I} \mathcal{O}_i$ are termwise dualizable, and similarly for $\mathcal{P}_i$ and $\mathcal{Q}_i$. Then the limit $\mathrm{lim}_{i \in I} S_i$ is also a composition square.
\end{lemma}
\begin{proof}
The claim about colimits is evident from the fact that the composition product preserves filtered colimits in each variable separately. (The same statement would have worked for sifted colimits.) For the claim about limits, first note that the hypothesis on unary terms implies that for each fixed arity $d$, the bar construction 
\[
\mathrm{colim}_{\mathbf{\Delta}^{\mathrm{op}}}(\mathcal{Q}_i \circ \mathcal{O}_i^{\circ\bullet} \circ \mathcal{P}_i)(d)
\] 
can be calculated by a finite colimit. It will therefore suffice to argue that for each $k$, the natural map
\[
(\mathcal{Q} \circ \mathcal{O}^{\circ k} \circ \mathcal{P})(d) \to \mathrm{lim}_{i \in I} (\mathcal{Q}_i \circ \mathcal{O}_i^{\circ k} \circ \mathcal{P}_i)(d)
\]
is an equivalence. This follows from the dualizability hypotheses we have imposed: indeed, for a dualizable object $M$ the functor $M \otimes -$ preserves limits.
\end{proof}

\begin{proof}[Proofs of \Cref{thm:LieAssCom} and \Cref{thm:specialcases}]
Starting from the composition square
\[
\begin{tikzcd}
\E_{k+m} \ar{d}{\beta}\ar{r}{\iota} & \E_{k+m+n} \ar{d}{\beta} \\
s^k\E_{m} \ar{r}{\iota} & s^k\E_{m+n}.
\end{tikzcd}
\]
of \Cref{thm:En}, we can set $m=0$ and take the colimit as $n$ goes to infinity to obtain a composition square
\[
\begin{tikzcd}
\E_{k} \ar{d}{\beta}\ar{r}{\iota} & \E_{\infty} \ar{d} \\
\mathbf{1} \ar{r}{\iota} & s^k\E_{\infty},
\end{tikzcd}
\]
where we have applied the identification $s^k \mathbf{1} \simeq \mathbf{1}$. To identify the right-hand vertical map we consider the commutative diagram (using Lemma \ref{lem:betaiota}):
\[
\begin{tikzcd}
\E_n \ar{r}{\iota}\ar{d}{\beta}\ar{dr}{\sigma} & \E_{n+1} \ar{r}{\iota}\ar{d}{\beta}\ar{dr}{\sigma}  & \E_{n+2} \ar{r}{\iota}\ar{d}{\beta}\ar{dr}{\sigma}  & \cdots \ar{r} & \E_{\infty} \ar{d} \\
s\E_{n-1} \ar{r} & s\E_n \ar{r}{\iota} & s\E_{n+1} \ar{r}{\iota} & \cdots \ar{r} & s \E_{\infty}.
\end{tikzcd}
\]
The right-hand vertical map is, by construction, the colimit of the vertical arrows $\beta$. The commutativity of the diagram shows that this arrow must be equivalent to the colimit over $n$ of the slanted morphisms $\sigma\colon \E_n \to s\E_n$, which is indeed the morphism $\sigma\colon \E_{\infty} \to s\E_{\infty}$ appearing in the statement of \Cref{thm:specialcases}. (Indeed, since the suspension $\sigma$ for an arbitrary operad $\mathcal{O}$ is defined by levelwise tensoring $\mathcal{O}$ with the suspension morphism for $\E_\infty$, it is clear that $\sigma$ is compatible with filtered colimits of operads.)

Applying the $k$-fold operadic desuspension to the composition square of \Cref{thm:En} yields a composition square
\[
\begin{tikzcd}
s^{-k}\E_{k+m} \ar{d}{\beta}\ar{r}{\iota} & s^{-k}\E_{k+m+n} \ar{d}{\beta} \\
\E_{m} \ar{r}{\iota} & \E_{m+n}.
\end{tikzcd}
\]
Now we take the limit as $k$ goes to $\infty$. Remark \ref{rmk:betaKD} allows us to identify the limit over $k$ of the maps $\beta\colon s^{-k} \E_k \to \mathbf{1}$ with the Koszul dual of the colimit over $k$ of the maps $\iota\colon \mathbf{1} \to \E_k$. Thus $\varprojlim_k s^{-k} \E_k$ is the Koszul dual of $\E_{\infty}$, which is the spectral Lie operad $\mathbb{L}$. (This identification of the inverse limit of shifted $\E_k$-operads with $\mathbb{L}$ was first obtained by Ching--Salvatore \cite{chingsalvatore}.) Therefore, the limit of the squares above becomes
\[
\begin{tikzcd}
s^{m}\mathbb{L} \ar{d}{\beta}\ar{r}{\sigma^n} & s^{n+m}\mathbb{L} \ar{d}{\beta} \\
\E_m \ar{r}{\iota} & \E_{m+n}.
\end{tikzcd}
\]
The identification of the top horizontal morphism with the $n$-fold suspension $\sigma^n$ is entirely analogous to the argument in the first half of this proof. Now taking $m=0$ gives the second composition square of \Cref{thm:specialcases}, whereas setting $m=1$ and taking the colimit as $n$ goes to $\infty$ gives the composition square of \Cref{thm:LieAssCom}. For this last part observe that for any operad $\mathcal{O}$, the colimit $\mathrm{colim}_n s^n \mathcal{O}$ along the suspension maps $\sigma$ is equivalent to $\mathcal{O}(1)$, interpreted as an operad concentrated in arity 1. Indeed, the case of general $\mathcal{O}$ is implied by the special case $\mathcal{O} = \E_\infty$. Observe that there is an identification of spectra $(s^n \E_{\infty})(k) = \mathbb{S}^{n(k-1)}$. For $k > 1$ the colimit is therefore arbitrarily connected, hence contractible.
\end{proof}

Recall that the morphism $\beta\colon s^n\mathbb{L} \to \E_n$ appearing in \Cref{thm:specialcases} induces a functor \[ \beta_!\colon \Alg_{s^n\mathbb{L}}(\Sp) \to \Alg_{\E_n}(\Sp)\] or, after identifying the $\infty$-categories $\Alg_{s^n\mathbb{L}}(\Sp)$ and $\Alg_{\mathbb{L}}(\Sp)$ via degree shifting, a functor 
\[
U_n\colon \Alg_{\mathbb{L}}(\Sp) \to \Alg_{\E_n}(\Sp)
\]
that we refer to as the $\E_n$-enveloping algebra functor. \Cref{thm:Un}, which describes $U_n$ as the composition $\mathrm{CE} \circ \Omega^n$, is now straightforward to deduce from our results:

\begin{proof}[Proof of \Cref{thm:Un}]
By Lemma \ref{lem:recognitioncomposition}, the first composition square of \Cref{thm:specialcases} gives a commutative square
\[
\begin{tikzcd}
\mathrm{Alg}_{\mathbb{L}}(\Sp) \ar{d}{\mathrm{CE}} & \mathrm{Alg}_{s^n\mathbb{L}}(\Sp) \ar{l}[swap]{(\sigma^n)^*} \ar{d}{\beta_!} \\
\Sp & \Alg_{\E_n}(\Sp) \ar{l}{\mathrm{forget}}.
\end{tikzcd}
\]
By \Cref{prop:suspensionmorphism}, the top arrow can be identified with the $n$-fold loop functor of Lie algebras through the following commutative diagram, where the vertical arrow is the $n$-fold shift:
\[
\begin{tikzcd}
\mathrm{Alg}_{s^n\mathbb{L}}(\Sp) \ar{r}{(\sigma^n)^*} \ar{d}{\simeq} & \mathrm{Alg}_{\mathbb{L}}(\Sp). \\
\mathrm{Alg}_{\mathbb{L}}(\Sp) \ar{ur}[swap]{\Omega^n_{\mathbb{L}}} & 
\end{tikzcd}
\]
This concludes the proof.
\end{proof}

We establish one further family of composition squares to be used in our discussion of the PBW theorem below. We will write $\E_k^{\vee}$ for the termwise Spanier--Whitehead dual of the (stable, nonunital) $\E_k$-operad and use the `self-duality' of the $\E_k$-operad established by Ching--Salvatore  \cite{chingsalvatore}, which in particular implies an equivalence of symmetric sequences
\[
\mathbf{1} \circ_{\E_k} \mathbf{1} \cong s^k\E_k^{\vee}.
\]

\begin{proposition}
\label{prop:compositionEnPBW}
There is a composition square as follows:
\[
\begin{tikzcd}
\E_{k+n} \ar{d}{\beta}\ar{r} & \mathbf{1} \ar{d} \\
s^k\E_n \ar{r} & s^k\E_k^{\vee}.
\end{tikzcd}
\]
\end{proposition}
\begin{proof}\
Consider the following diagram:
\[
\begin{tikzcd}
\E_k \ar{d}\ar{r}{\iota} & \E_{k+n} \ar{d}{\beta} \ar{r} & \mathbf{1} \ar{d} \\
\mathbf{1} \ar{r} & s^k\E_n \ar{r} & s^k\E_n \circ_{\E_{k+n}} \mathbf{1}.
\end{tikzcd}
\]
The left square is a composition square by the case $m=0$ of \Cref{thm:En}, the right square is a composition square by construction. Composing the two shows that the outer rectangle is a composition square, giving an equivalence $s^k\E_n \circ_{\E_{k+n}} \mathbf{1} \simeq \mathbf{1} \circ_{\E_k} \mathbf{1}$. The conclusion now follows from the display preceding the statement of the proposition.
\end{proof}

So far we have focused on commutative squares of operads that are `composition squares'. A natural question is whether our squares, e.g.\ the one of \Cref{thm:LieAssCom} involving the Lie, associative, and commutative operads, are also pushout squares in the $\infty$-category of operads. In the 1-category of operads (in abelian groups, say) the corresponding square \emph{is} indeed a pushout, as one may verify by using the fact that the three operads involved have straightforward presentations with generators in arity 2 and relations in arity 3. However, this turns out to be a bit of an accident:

\begin{proposition}
\label{prop:notpushouts}
The square
\[
\begin{tikzcd}
s\mathbb{L} \ar{d}\ar{r} & \mathbf{1} \ar{d} \\
\E_1 \ar{r} & \E_{\infty}
\end{tikzcd}
\]
is not a pushout in the $\infty$-category of operads in $\Sp$. As a consequence, it cannot be the case that the square of \Cref{thm:En} is a pushout for all \(k\) and \(n\) when \(m=0\).
\end{proposition}
\begin{proof}
The operadic bar construction $\mathrm{Bar}\colon \mathrm{Op}(\Sp) \to \sSeq(\Sp)$ is a colimit-preserving functor. (Generally, the bar construction of augmented $\E_1$-algebras is a left adjoint functor to $\E_1$-coalgebras; composing with the forgetful functor and specializing to operads gives the claim.) This bar construction turns the square of the proposition into the following:
\[
\begin{tikzcd}
s\E_{\infty}^{\vee} \ar{d}\ar{r} & \mathbf{1} \ar{d} \\
s\E_1^{\vee} \ar{r} & \mathbb{L}^{\vee}
\end{tikzcd}
\]
(This happens to be precisely the termwise Spanier--Whitehead dual of the square we started with, up to operadic suspension.) If the square of the proposition were a pushout of operads, then this square would be a pushout of symmetric sequences. However, this is not the case; indeed, if it were a pushout then for $n \geq 2$ there would be a cofibre sequence
\[
\E_{\infty}^{\vee}(n) \to \E_1^{\vee}(n) \to s^{-1}\mathbb{L}^{\vee}(n).
\]
The homology of these spectra is concentrated in degree zero; their integral homology is finitely generated and free of ranks $1$, $n!$, and $(n-1)!$ respectively. The alternating sum of these numbers does not equal zero as soon as $n \geq 3$. This proves the first claim of the proposition. If the squares of \Cref{thm:En} were pushouts for all \(k\) and \(n\) when \(m=0\), then that would imply (by pasting countably many such squares) that the square of the proposition is a pushout, which we have just ruled out.
\end{proof}

\section{Relation to the Poincar\'{e}--Birkhoff--Witt theorem}

Given a classical Lie algebra $\mathfrak{g}$, let us 
denote its  universal enveloping algebra by $U(\mathfrak{g})$. Concretely, it is the quotient of the tensor algebra $T( \mathfrak{g})$ by the two-sided ideal generated by elements   $$ x \otimes y - y \otimes x - [x,y], \ x,y \in \mathfrak{g}.$$
This ideal is not homogeneous, so the natural grading on $T( \mathfrak{g})$ does \textit{not} descend to a grading on $U(\mathfrak{g})$. However, $U(\mathfrak{g}) $ inherits an increasing filtration whose $n^{th}$ piece $F_n(U(\mathfrak{g})) $ consists of all images of elements $x_1 \otimes \ldots \otimes x_k$ with $k\leq n$   under the quotient map $T(\mathfrak{g})\rightarrow U(\mathfrak{g})$.

The canonical map $\mathfrak{g}\rightarrow  U(\mathfrak{g})$ lands in the first filtered piece $F_1 (U(\mathfrak{g}))$, and we therefore obtain a map of modules $$ \mathfrak{g}\rightarrow \gr_1( U(\mathfrak{g})).$$
As the associative product on  $U(\mathfrak{g})$ descends to a commutative product on $\gr( U(\mathfrak{g})) = \bigoplus_i  \gr_i (U(\mathfrak{g}))$, we can induce up to obtain a homomorphism of commutative algebras $\mathrm{Sym} ( \mathfrak{g})\rightarrow \gr ( U(\mathfrak{g})) .$
\begin{theorem}[Poincar\'{e}--Birkhoff--Witt theorem]
	This induced morphism $$\mathrm{Sym} ( \mathfrak{g})\rightarrow \gr ( U(\mathfrak{g})) $$
	is an isomorphism.
\end{theorem}
\noindent \mbox{We will   prove an analogue of this result using 
 the relative composition products   computed above.  }
\begin{remark} We will adopt the general strategy in the proof of  \cite[6.5.2.6]{gaitsgory2017study} to our setting. 
	However, our treatment  differs from the one in [op.\ cit.] as the latter operates over a ground field $k$ and uses the commutative $k$-algebra $k[\epsilon]/(\epsilon^2)$ to generate the required filtrations. 
\end{remark}

\begin{notation} 
	Write $\mathbb{N} $ for (the nerve of) the poset of nonnegative integers   and   $\mathbb{N} ^{\mathrm{disc}}$ for  (the nerve of) the discrete category with the same objects and only  identity morphisms.
	\end{notation}
	\begin{notation}[Filtered and graded spectra]
The $\infty$-categories of (nonnegatively) filtered and graded spectra are defined as $$\Sp^{\Fil} := \Fun( \mathbb{N}  , \Sp) \ \ \ \mbox{and}\ \ \ \Sp^{\gr} := \Fun( \mathbb{N} ^{\mathrm{disc}}, \Sp),$$ respectively. Day convolution (cf.\ \cite{glasman2016day},\cite[2.2.6]{HA}) equips these $\infty$-categories with   symmetric monoidal structures satisfying 
$$ (X \otimes Y)_n \simeq \mycolim{i+j \leq n} (X_i \otimes Y_j) \ \ \ \mbox{and}\ \ \   \ (X \otimes Y)_n  \simeq  \bigoplus_{i+j =n} (X_i \otimes Y_j). $$
There are natural functors 
	$$ \Sp \xrightarrow{c} \Sp^{\Fil}   \ \ \ \ \ \ \ \ 
	  \Sp^{\Fil}  \xrightarrow{\gr} \Sp^{\gr}    
	   \ \ \ \  \ \ \ \  \Sp^{\Fil}  \xrightarrow{\colim} \Sp.$$
Informally,  $c$ sends  $X \in \Sp$ to the constant object $(X \xrightarrow{\id} X  \xrightarrow{\id } \ldots)  \in \Sp^{\Fil}$,
its left adjoint $\colim$ sends $(X_0 \rightarrow X_1 \rightarrow \ldots)$ to $\colim_i X_i$, and 
 the functor $\gr$ sends a filtered spectrum $(X_0 \rightarrow X_1 \rightarrow \ldots )$ to the graded spectrum $(X_0, X_1/X_0, X_2/X_1,  \ldots ) $; we refer to \cite{gwilliam2018enhancing} for   formal constructions.

The functor $c$ is fully faithful and  evidently refines to a symmetric monoidal functor. By   \cite[Proposition A]{haugseng2023lax}, its left adjoint $\colim$ acquires an oplax symmetric monoidal structure, which is in fact symmetric monoidal as $\otimes$ commutes with colimits.
The functor $\gr$ refines to a symmetric  monoidal functor by Theorem 1.13 of \cite{gwilliam2018enhancing}. 
	 \end{notation}

\Cref{sec:symseq} immediately generalises  to  filtered and graded spectra -- in particular, both  $\sSeq(\Sp^{\Fil})$ and  $\sSeq(\Sp^{\gr})$  carry composition products,
which one uses to define  operads  \mbox{and algebras over them.}
 
The  functors
$c$, $\gr$,   and $\colim$ induce 
functors  $\sSeq(\Sp) \rightarrow \sSeq(\Sp^{\Fil}) $,  $\sSeq(\Sp^{\Fil}) \rightarrow \sSeq(\Sp^{\gr})$, and  $\sSeq(\Sp^{\Fil}) \rightarrow \sSeq(\Sp^{})$.
These 
 all refine to  monoidal functors with respect to  the composition products as
 $c$, $\gr$,   and $\colim$ are
 colimit-preserving and symmetric monoidal. Indeed, this is  immediate by unravelling  the 
  definition of the composition products via endofunctors of symmetric monoidal presentable $\infty$-categories.

	Let  us now fix an operad $\mathcal{O} \in \Op(\Sp) = \Alg(\sSeq(\Sp))$   in spectra. Using the 
	monoidal functors  above, we obtain corresponding $\infty$-operads in $ \Sp^{\Fil}$ and $\Sp^{\gr}$, which we   denote by the same name. We
	write 	  $ \Alg_{\mathcal{O}}(\Sp^{\Fil})$ and $ \Alg_{\mathcal{O}}(\Sp^{\gr})$ for the 
	corresponding  $\infty$-categories of $\mathcal{O}$-algebras, and     will refer to them as  filtered and graded $\mathcal{O}$-algebras, respectively. 
	\begin{remark} Note that  a filtered (graded) $\mathcal{O}$-algebra is an $\mathcal{O}$-algebra in filtered (graded) objects, rather than a filtered (graded) object in $\mathcal{O}$-algebras.\end{remark}
	  
Once more using that $c, \gr,$ and $ \colim$ induce monoidal functors for the composition products on symmetric sequences,  
we obtain induced  functors 
	$$\Alg_{\mathcal{O}}( \Sp )\xrightarrow{c} \Alg_{\mathcal{O}}(\Sp^{\Fil} )  \ \ \ \ \ \ \ \ 
\Alg_{\mathcal{O}}(\Sp^{\Fil} ) \xrightarrow{\gr}\Alg_{\mathcal{O}}( \Sp^{\gr}    )
\ \ \ \  \ \ \ \  \Alg_{\mathcal{O}}(\Sp^{\Fil} ) \xrightarrow{\colim} \Alg_{\mathcal{O}}(\Sp). $$

\begin{construction}[Adding filtration] There is also a more interesting functor from $\mathcal{O}$-algebras to filtered $\mathcal{O}$-algebras: given an 
$\mathcal{O}$-algebra $X$, we will equip the 
filtered spectrum 
$$( 0  \rightarrow X  \xrightarrow{\id} X  \xrightarrow{\id}  \ldots)\ $$
with an 
$\mathcal{O}$-algebra structure. 

To this end, consider the  functors 
$$\mathrm{ev}_0,  \colim: \Sp^{\Fil}   \rightarrow \Sp  $$
given by evaluating at $0$ and taking the colimit, respectively.
As both refine to symmetric monoidal colimit-preserving functors,   we obtain induced functors 
$   \Alg_{\mathcal{O}}( \Sp^{\Fil}  ) \rightarrow  \Alg_{\mathcal{O}}( \Sp   )$ on the level of $\mathcal{O}$-algebras, which we denote by the same names, and a natural transformation $\mathrm{ev}_0 \rightarrow  \colim$.
Writing $0:  \Alg_{\mathcal{O}}( \Sp^{\Fil}  ) \rightarrow  \Alg_{\mathcal{O}}( \Sp   )$ for the constant functor on the zero $\mathcal{O}$-algebra, we construct a functor $$L: \Alg_{\mathcal{O}}( \Sp^{\Fil}  ) \rightarrow \Alg_{\mathcal{O}}( \Sp   ) $$
sending $X$ to the colimit (in $\mathcal{O}$-algebras) of the span $ (0 \leftarrow \mathrm{ev}_0(X) \rightarrow \colim(X) )$.

We note that $L$ sits in a commutative square
\[
\begin{tikzcd}
	\Alg_{\mathcal{O}}(\Sp^{\Fil}) \ar{r}{L} &  	\Alg_{\mathcal{O} }(\Sp)        \\
	\Sp^{\Fil} \ar{r}{} \ar{u}& \Sp \ar{u}
\end{tikzcd}
\]
where the vertical arrows are free functors and the lower horizontal functor sends a filtered spectrum $X$ to the cofibre $\cof\left(X_0 \rightarrow \colim_n X_n\right) $. Passing to right adjoints, we obtain a functor 
$$c_1:\Alg_{\mathcal{O}}( \Sp ) \rightarrow \Alg_{\mathcal{O}}( \Sp^{\Fil}  ). $$
lifting the functor  
sending $X$  to the filtered spectrum 
$( 0  \rightarrow X  \xrightarrow{\id} X  \xrightarrow{\id}  \ldots)$
\mbox{starting in degree $1$.}

We obtain a a commutative diagram
\[
\begin{tikzcd}
	\Alg_{\mathcal{O}}(\Sp) \ar{d}\ar{r}{c_1} &  	\Alg_{\mathcal{O} }(\Sp^{\Fil})  \ar{d} \ar{d}\ar{r}{\colim} & \Alg_{\mathcal{O} }(\Sp )  \ar{d}    \\
	\Sp \ar{r}{c_1} & \Sp^{\Fil}\ar{r}{\colim} & \Sp.
\end{tikzcd}
\]
where the vertical arrows are   forgetful functors and the horizontal arrows compose to the identity. In other words, we have equipped every $\mathcal{O}$-algebra with a natural filtration.
\end{construction}
Let us now assume that the operad $\mathcal{O}$ is reduced, i.e.\ that $\mathcal{O}(0) \simeq 0$ and $\mathcal{O}(1) \simeq \mathbb{S}^0$. We obtain a map of operads $\mathcal{O} \rightarrow \mathbf{1}$, where $\mathbf{1}$ is the identity operad. Pulling back along it gives a functor $$\triv: \Sp^{\gr} \rightarrow \Alg_{\mathcal{O}}(\Sp^{\gr}).$$
\begin{proposition}\label{deg1triv}
	If $X \in \Alg_{\mathcal{O}}( \Sp^{\gr})$ is concentrated in degree $1$, then there is an equivalence of graded $\mathcal{O}$-algebras 
	$$ \triv (\forget(X)) \simeq X, $$
where $\forget: \Alg_{\mathcal{O}}(\Sp) \rightarrow \Sp$ is the forgetful functor.
\end{proposition}
\begin{proof}
	We will need to discuss the formalisation of monoidal and tensored $\infty$-categories in Section 4  of \cite{HA}. 
	To this end, we will use the term \textit{$\infty$-operad} in the sense of \cite[2.1.1.10]{HA}, highlighting that it is distinct from (but closely related to)   the term \textit{operad} in  \Cref{operadsinsp}.
	
	We invite the reader to recall the $\infty$-operads $\mathcal{A}\mathrm{ssoc}^{\otimes}$ and $\mathcal{L}\mathcal{M}^{\otimes}$ constructed in \cite[Definition 
	4.1.1.3]{HA} and \cite[Definition 
    4.2.1.7]{HA}, respectively.
 Monoidal $\infty$-categories are formalised as cocartesian fibrations of $\infty$-operads to  $\mathcal{A}\mathrm{ssoc}^{\otimes}$. The $\infty$-operad $\mathcal{L}\mathcal{M}^{\otimes}$
 has two colours $a$ (for algebra) and $m$ (for module), and contains 
 $\mathcal{A}\mathrm{ssoc}^{\otimes}$.
 Cocartesian fibrations of $\infty$-operads $\mathcal{C}^{\otimes }\rightarrow \mathcal{L}\mathcal{M}^{\otimes}$ correspond to pairs of a monoidal $\infty$-category $q^{-1} (a) $  and 
 a  left-$q^{-1} (a)$-tensored 
  $\infty$-category $q^{-1} (m)$; the monoidal structure on 
    $q^{-1} (a) $ is formally given  by 
     $\mathcal{A}\mathrm{ssoc}^{\otimes} \times_{\mathcal{L}\mathcal{M}^{\otimes}}\mathcal{C}^{\otimes }\rightarrow \mathcal{A}\mathrm{ssoc}^{\otimes}$.
    
 Write  $p: \sSeq(\Sp^{\gr})^{\otimes } \rightarrow \mathcal{A}\mathrm{ssoc}^{\otimes}$
  for the cocartesian fibrations  encoding the monoidal composition product $\circ$ on  symmetric sequences in graded  spectra.
The  composition product    makes   $\sSeq(\Sp^{\gr})$  
left-tensored   over itself, and 
  \cite[Example 4.2.1.16]{HA} constructs the corresponding 
  cocartesian fibration  $q: \mathcal{C}^{\otimes} \rightarrow 
  \mathcal{L}\mathcal{M}^{\otimes}$ with $q^{-1}(a)  \simeq  \sSeq(\Sp^{\gr}) \simeq q^{-1}(m)$ and $q \times_{\mathcal{L}\mathcal{M}^{\otimes}} {\mathcal{A}\mathrm{ssoc}^{\otimes}} = p$.
  The underlying $\infty$-category  of $\mathcal{C}^{\otimes}$ is given by $\sSeq(\Sp^{\gr})  \coprod \sSeq(\Sp^{\gr}) = p^{-1}(a) \coprod  p^{-1}(m)$.
  
Let $\sSeq(\Sp)({1}) \subset \sSeq(\Sp)({+}) \subset \sSeq(\Sp^{\gr}) $ be the full subcategories spanned by   symmetric sequences $X$  with $X(n) \simeq  0  $ for $n\neq 1$ or $X(n) \simeq  0  $ for $n= 0$, respectively \mbox{(in graded degree $0$).}
Write $\Sp_{ 1}^{\gr} \subset \Sp_{+}^{\gr} \subset \sSeq(\Sp^{\gr}) $ for the full subcategories spanned by  spectra   in graded degree $1$ or by  spectra  in graded degree $>0$ , respectively (interpreted as symmetric sequences).
Let $\mathcal{C}_+^{\otimes } \subset \mathcal{C}^{\otimes}$ denote the 
 full subcategory  of $\mathcal{C}^{\otimes}$   spanned by those     $X_1 \oplus \ldots \oplus X_n$ for which all $X_i$  over $a$ or $m$ belong to $\sSeq(\Sp)({+})$ or  $\Sp^{\gr}_{+}$, respectively.
 Similarly,  let $\mathcal{C}_1^{\otimes } \subset \mathcal{C}_+^{\otimes}$ be  spanned by those    $X_1 \oplus \ldots \oplus X_n$ for which all $X_i$   over $a$ or $m$ belong to $\sSeq(\Sp)({1})$ or  $\Sp^{\gr}_1$.
 
 By \cite[Proposition 2.2.1.1]{HA}, the composite $\mathcal{C}_+^{\otimes} \hookrightarrow 
 \mathcal{C}^{\otimes} \rightarrow \mathcal{L}\mathcal{M}^{\otimes}$ is a cocartesian fibration of $\infty$-operads
 and $\mathcal{C}_+^{\otimes} \rightarrow 
 \mathcal{C}^{\otimes}$ preserves cocartesian morphisms.
The operad $\mathcal{O}$ gives an operadic section of $p$, which factors over 
$\sSeq(\Sp)(+)^{\otimes} := \mathcal{A}\mathrm{ssoc}^{\otimes}  \times_{\mathcal{L}\mathcal{M}^{\otimes}}\mathcal{C}_+^{\otimes}$ since $\mathcal{O}$ is reduced and in \mbox{degree $0$.} Writing $s_{\mathcal{O}}:\mathcal{A}\mathrm{ssoc}^{\otimes}  \rightarrow \sSeq(\Sp)(+)^{\otimes} $ for the corresponding operadic section, the full subcategory   $ \Alg_{\mathcal{O}}(\Sp^{\gr})_{+} \subset \Alg_{\mathcal{O}}(\Sp^{\gr})  $ of graded $\mathcal{O}$-algebras in positive degrees
is given by the $\infty$-category of operadic sections of $\mathcal{C}_+^{\otimes}  \rightarrow \mathcal{L}\mathcal{M}^{\otimes}$
whose pullback along $\mathcal{A}\mathrm{ssoc}^{\otimes} \rightarrow \mathcal{L}\mathcal{M}^{\otimes}$ \mbox{recovers $s_{\mathcal{O}}$.} In the notation of \cite[2.1.3.1]{HA}, we have $$\Alg_{\mathcal{O}}(\Sp^{\gr})_{+} = \Alg_{/\mathcal{L}\mathcal{M}^{\otimes}}(\mathcal{C}_+^{\otimes} ) \times_{\Alg_{/\mathcal{A}\mathrm{ssoc}^{\otimes}}(\sSeq(\Sp)(+) ^{\otimes} )} \{\mathcal{O}\}.$$

Observe that both $\sSeq(\Sp)(1) \subset \sSeq(\Sp)(+)$
and  $\Sp^{\gr}_1 \subset \Sp^{\gr}_{+}$  are localising subcategories, and that the   localisation functors are  compatible with the $  \mathcal{L}\mathcal{M}^{\otimes}$-monoidal structure in the sense of \cite[Definition 2.2.1.6]{HA}.
By Proposition 2.2.1.9 in \cite{HA},
the inclusion   
$ \mathcal{C}_1^{\otimes}  \subset  \mathcal{C}_+^{\otimes}$ is a map of $\infty$-operads and moreover admits a left adjoint $L^{\otimes}$.

  Postcomposing operadic sections $\mathcal{L}\mathcal{M}^{\otimes} \rightarrow \mathcal{C}_+^{\otimes}$ with this adjunction, we obtain 
an adjunction $$\Alg_{\mathcal{O}} (\Sp^{\gr})_{+}  \leftrightarrows  \Alg_{ \mathbf{1}} (\Sp^{\gr}_1),$$
as the identity operad $ \mathbf{1}$ is the localisation of $\mathcal{O}$.
The value of the unit of this adjunction on a given  $A\in \Alg_{\mathcal{O}} (\Sp^{\gr})_{+} $ is  the localisation map $ A \rightarrow \triv(A_1), $
where  $ \triv(A_1)$ has underlying object  $A_1$ (considered as a graded spectrum concentrated in degree $1$)  and   structure map  given by \mbox{$\mathcal{O} \circ A_1 \rightarrow  \mathbf{1} \circ A_1 =  A_1$.}
If $A$ lies in degree $1$, then this is an equivalence as equivalences are detected on underlying objects.
The above adjunction therefore restricts to an equivalence
$\Alg_{\mathcal{O}} (\Sp^{\gr})_{1}  \xrightarrow{\simeq} \Alg_{ \mathbf{1}} (\Sp^{\gr}_1) \simeq \Sp$ between graded $\mathcal{O}$-algebras concentrated in degree $1$ and spectra.
\end{proof}
\begin{corollary}\label{cortrivgrad}
	Given an $\mathcal{O}$-algebra $X$, there is a natural  equivalence of graded $\mathcal{O}$-algebras   $$ \gr (c_1(X)) \simeq \triv(\forget(X)).$$
\end{corollary}
\begin{proof}This follows immediately from \Cref{deg1triv} as $\gr(c_1(X))$ is concentrated in degree $1$.
\end{proof}

Fix a morphism of spectral operads $ \beta:  \mathcal{P} \rightarrow \mathcal{Q}.$
 The induced functor
$\beta^\ast: \Alg_{\mathcal{Q}}(\Sp) \rightarrow  \Alg_{\mathcal{P}}(\Sp)$ admits a left adjoint 
 $\beta_!: \Alg_{\mathcal{P}}(\Sp) \rightarrow  \Alg_{\mathcal{Q}}(\Sp),$ which we call the pushforward along $\beta$, and  we denote the corresponding functors on filtered and graded algebras by the same name. Note that they   are compatible with the functors $\gr$ and $\colim$.

\begin{proposition}\label{pushgr}
Given a $\mathcal{P}$-algebra $X$, the $\mathcal{Q}$-algebra $\beta_!(X)$ admits an exhaustive filtration  $\beta_!(c_1(X))$ 
 with associated graded  $\beta_{!}(\triv(\forget(X)))$.
\end{proposition}
\begin{proof} 
To verify that this filtration is indeed exhaustive, we compute 
	$$  \colim  \beta_!(c_1(X)))\simeq 
\beta_!(\colim c_1(X))) 	
 \simeq 	\beta_!(X). $$
Using \Cref{cortrivgrad}, we obtain
$ \gr(\beta_!(c_1(X))) ) \simeq  \beta_!(\gr(c_1(X))) ) \simeq \beta_{!}(\triv(\forget(X))).$
\end{proof}
We will now establish analogues of the PBW theorem.
Recall that 
the universal enveloping algebra functor $$U: \Alg_{ \mathbb{L} }(\Sp) \rightarrow\Alg_{ \mathbb{E}_1}(\Sp)$$ 
is  obtained by composing the equivalence $\Alg_{s^n\mathbb{L}}(\Sp)$ and $\Alg_{\mathbb{L}}(\Sp)$  with the 
pushforward $\beta_!$ along the  morphism of operads 
$\beta:  s\mathbb{L} \rightarrow  \mathbb{E}_1,$
which is Koszul dual to the natural  inclusion $\mathbb{E}_1 \rightarrow \mathbb{E}_\infty$.

We can now prove \Cref{PBWfull}, the PBW theorem for spectral Lie algebras.
\begin{proof}[Proof of \Cref{PBWfull}]
 \Cref{pushgr} gives an exhaustive filtration $ U (c_1(\mathfrak{g}))$ with associated graded $ U (\triv (\forget(\mathfrak{g}))). $
For $X = \forget(\mathfrak{g})$, we have equivalences 
\[U(\triv(X)) \simeq U (|\Barc_\bullet(  \mathbb{L},   \mathbb{L}, \triv(X))|) \simeq |\Barc_\bullet(\E_1,s \mathbb{L}, \triv(\Sigma^{-1}X))| \simeq (\mathbb{E}_1 \circ_{s \mathbb{L}} \mathbf{1})(\Sigma^{-1}X),\] and so the claim  follows from \Cref{PBWtriv}.
\end{proof}

We can also construct envelopes of $\E_n$-algebras as follows:
\begin{definition}[Relative  envelope] \label{relenv}
	Given   $0 \leq m \leq n$, 
	the \textit{relative  enveloping algebra functor} $$ U_{n,m}: \Alg_{ \mathbb{E}_n }(\Sp) \rightarrow\Alg_{ s^{n-m}\mathbb{E}_{m} }(\Sp)$$ is  given by the pushforward along the morphism of operads 
	$  \mathbb{E}_n \rightarrow s^{n-m} \mathbb{E}_{m}$,
	which is Koszul dual to the inclusion $\mathbb{E}_{m} \rightarrow \mathbb{E}_n$.\end{definition}

We conclude by proving  the PBW theorem for $\mathbb{E}_n$-algebras.
\begin{proof}[Proof of \Cref{PBWEn}]
	This follows from \Cref{prop:compositionEnPBW} by the same argument as in the proof of \Cref{PBWfull}.
\end{proof}
\bibliographystyle{amsalpha}
\bibliography{biblio}

\end{document}